\documentclass[reqno, 12pt]{amsart}
\usepackage{amsmath}
 \usepackage{amssymb}
\usepackage{amsthm}
\usepackage{times}
\usepackage{latexsym}
\usepackage[mathscr]{eucal}
\usepackage{amscd}
\numberwithin{equation}{section}
 
  \newtheorem{theorem}{Theorem}[section]
  \newtheorem{proposition}[theorem]{Proposition}
  
  \newtheorem{corollary}[theorem]{Corollary}

  \newtheorem{definition}[theorem]{Definition}
  \newtheorem{example}[theorem]{Example}

\title[On twisted Riemannian extensions associated with Szab\'o metrics]{On twisted Riemannian extensions associated with Szab\'o metrics}
\author[Abdoul Salam Diallo, Silas Longwap, Fortun\'{e} Massamba]{Abdoul Salam Diallo*, Silas Longwap**, Fortun\'{e} Massamba***}

\newcommand{\acr}{\newline\indent}

\address{\llap{*\,} School of Mathematics, Statistics and Computer Science\acr
 University of KwaZulu-Natal\acr
 Private Bag X01, Scottsville 3209\acr
South Africa  \acr
and \acr
Universit\'e Alioune Diop de Bambey\acr
UFR SATIC, D\'epartement de Math\'ematiques\acr
B. P. 30, Bambey, S\'en\'egal}
\email{Diallo@ukzn.ac.za, abdoulsalam.diallo@uadb.edu.sn}
\address{\llap{**\,}  School of Mathematics, Statistics and Computer Science\acr
 University of KwaZulu-Natal\acr
 Private Bag X01, Scottsville 3209\acr
South Africa} \email{longwap4all@yahoo.com}
\address{\llap{***\,} School of Mathematics, Statistics and Computer Science\acr
 University of KwaZulu-Natal\acr
 Private Bag X01, Scottsville 3209\acr
South Africa} \email{massfort@yahoo.fr, Massamba@ukzn.ac.za}
\thanks{}



\begin{document}
 
\begin{abstract} 
Let $M$ be an $n$-dimensional manifold with a torsion free affine connection $\nabla$ and let $T^*M$ be the cotangent bundle. 
In this paper, we consider some of the geometrical aspect of a twisted Riemannian extension which provide a link between
the affine geometry of $(M,\nabla)$ and the neutral signature pseudo-Riemannian geometry of $T^*M$. We investigate the
spectral geometry of the Szab\'o operator on $M$ and on $T^*M$.  
\end{abstract} 
 
\maketitle

\section{Introduction}

Let $M$ be an $n$-dimensional manifold with a torsion free affine connection and let $T^* M$ be the cotangent bundle. 
In \cite{PattersonWalker}, Patterson and Walker introduced the notion of \textit{Riemann extensions} and showed how a 
pseudo-Riemannian metric can be given to the $2n$-dimensional cotangent bundle of an $n$-dimensional manifold with given 
non-Riemannian structure. They shows that Riemann extension provides a solution of the general problem of embedding a
manifold $M$ carrying a given structure in a manifold $N$ carrying another structure, the embedding being carried out in 
such a way that the structure on $N$ induces in a natural way the given structure on $M$. The Riemann extension  can be 
constructed with the help of the coefficients of the affine connection.

The Riemann extensions which are pseudo-Riemannian metrics of neutral neutral signature shown its importance in relation to the Osserman manifolds \cite{gar}, Walker manifolds \cite{broz} and non-Lorentzian geometry. In \cite{broz}, the authors generalize the usual Riemannian extensions to the so-called \textit{twisted Riemannian extensions}. The latter is also called \textit{deformed Riemannian extension} (see \cite{GarciaGilkeyNikcevicLorenzo} for more details).  In \cite{broz,GarciaGilkeyNikcevicLorenzo}, the authors studied the spectral geometry of the Jacobi operator and skew-symmetric curvature operator both on $M$ and on $T^{*}M$. The results on these operators are detailed, for instance, in \cite[Theorem 2.15]{GarciaGilkeyNikcevicLorenzo}. 

In this paper, we shall consider some of the geometric aspects of twisted Riemannian extensions and we will investigate the spectral geometry of the Szab\'o operator on $M$ and on $T^*M$. Note that the Szab\'o operator has not been deeply studied like the Jacobi and skew-symmetric curvature operators. 

Our paper is organized as follows. In the section \ref{prel}, we recall some basic definitions and results on the classical Riemannian extension and the twisted Riemannian extension developed in the books \cite{broz,GarciaGilkeyNikcevicLorenzo}. Finally in section \ref{Szabo}, we investigates the spectral geometry of the Szab\'o operator on $M$ and on $T^*M$, and we construct two examples of pseudo-Riemannian Szab\'o metrics of signature $(3,3)$, using the classical and twisted Riemannian extensions, whose Szab\'o operators are nilpotent.  

Throughout this paper, all manifolds, tensors fields and connections are always assumed to be differentiable of class $\mathcal{C}^{\infty}$.

\section{Twisted Riemannian extension}\label{prel}

Let $(M,\nabla)$ be an $n$-dimensional affine manifold and $T^* M$ be its cotangent bundle and let 
$\pi: T^* M \rightarrow M$ be the natural projection defined by $\pi(p,\omega)=p\in M$ and $(p,\omega)\in T^*M$.
A system of local coordinates $(U,u_i), i=1,\cdots,n$ around $p\in M$ induces a system of local
coordinates $(\pi^{-1}(U), u_i, u_{i'}=\omega_i), i'=n+i=n+1,\cdots,2n$ around $(p,\omega) \in T^*M$,
where $u_{i'}=\omega_i$ are components of covectors $\omega$ in each cotangent space $T^{*}_{p} M$,
$p\in U$ with respect to the natural coframe $\{du^i\}$. If we use the notation $\partial_i =\frac{\partial}{\partial u_i}$
and $\partial_{i'}=\frac{\partial}{\partial \omega_i}, i=i,\cdots,n$ then at each point $(p,\omega)\in T^* M$,
its follows that
$$
 \{(\partial_1)_{(p,\omega)},\cdots,(\partial_n)_{(p,\omega)},(\partial_{1'})_{(p,\omega)},\cdots,
(\partial_{n'})_{(p,\omega)} \},
$$
is a basis for the tangent space $(T^* M)_{(p,\omega)}$.
 
For each vector field $X$ on $M$, define a function $\iota X: T^{*}M\longrightarrow \mathbb{R}$ by 
$$
\iota X(p,\omega) = \omega(X_{p}).
$$ 
This function is locally expressed by,  
$$ 
\iota X (u_i,u_{i'})= u_{i'} X^i,
$$ 
for all $X=X^{i}\partial_{i}$. Vector fields on $T^*M$ are characterized by their actions on functions $\iota X$. The complete lift $X^C$ of a vector 
field $X$ on $M$ to $T^*M$ is characterized by the identity
$$
X^C (\iota Z) = \iota [X,Z], \quad \mbox{for all}\quad Z\in \Gamma(TM). 
$$
Moreover, since a $(0,s)$-tensor field on $M$ is characterized by its evaluation on complete lifts of vector 
fields on $M$, for each tensor field $T$ of type $(1,1)$ on $M$, we define a $1$-form $\iota T$ on $T^*M$ which is 
characterized by the identity
$$
\iota T(X^C) = \iota (TX). 
$$
For more details on the geometry of cotangent bundle, we refer to the book of Yano and Ishihara \cite{yano}.

Let $\nabla$ be a torsion free affine connection on an $n$-dimensional affine manifold $M$. The \textit{Riemannian extension} $g_{\nabla}$ is the pseudo-Riemannian metric on $N$ of neutral signature
$(n,n)$ characterized by the identity \cite{broz,GarciaGilkeyNikcevicLorenzo}
$$
g_{\nabla}(X^C,Y^C) = -\iota (\nabla_X Y + \nabla_Y X). 
$$
In the locally induced coordinates $(u_i,u_{i'})$ on $\pi^{-1}(U)\subset T^* M$, the
Riemannian extension is expressed by
\begin{eqnarray}
 g_{\nabla}= 
\left(
      \begin{array}{cc}
      -2u_{k'}\Gamma^{k}_{ij}&\delta^{j}_{i}\\
      \delta^{j}_{i}&0
         \end{array}
\right),
\end{eqnarray}
with respect to $\{\partial_1,\cdots,\partial_n,\partial_{1'},\cdots,\partial_{n'}\} (i,j,k=1,\cdots,n;k'=k+n)$, where 
$\Gamma^{k}_{ij}$ are the Christoffel symbols of the torsion free affine connection $\nabla$ with respect to $(U,u_i)$ on $M$.
Some properties of the affine connection $\nabla$ can be investigated by means of the corresponding properties of the 
Riemannian extension $g_{\nabla}$. For instance, $(M,\nabla)$ is locally symmetric if and only if $(T^*M, g_{\nabla})$
is locally symmetric \cite{GarciaGilkeyNikcevicLorenzo}. Furthermore $(M,\nabla)$ is projectively flat if and only if $(T^*M, g_{\nabla})$ is locally 
conformally flat (see \cite{calvino2} for more details and references therein).

Let $\phi$ be a symmetric $(0,2)$-tensor field on $M$. The \textit{twisted Riemannian extension} is the neutral signature metric on $T^*M$ given by \cite{broz, GarciaGilkeyNikcevicLorenzo}
\begin{eqnarray}
 g_{(\nabla,\phi)}= 
\left(
      \begin{array}{cc}
      \phi_{ij}(u)-2u_{k'}\Gamma^{k}_{ij}&\delta^{j}_{i}\\
      \delta^{j}_{i}&0
         \end{array}
\right),
\end{eqnarray}
with respect to $\{\partial_1,\cdots,\partial_n,\partial_{1'},\cdots,\partial_{n'}\}, (i,j,k=1,\cdots,n;k'=k+n)$,
where $\Gamma^{k}_{ij}$ are the Christoffel symbols of the torsion free affine connection $\nabla$ with respect to $(U,u_i)$.

As an example of twisted Riemannian extension metrics, we have the Walker metrics. The latter is detailed as follows. We say that a neutral signature pseudo-Riemannian metric $g$ of a $2n$-dimensional manifold is a \textit{Walker metric} if, locally, we have
\begin{eqnarray*}
 g= 
\left(
      \begin{array}{cc}
      B&I_n\\
      I_n&0
         \end{array}
\right).
\end{eqnarray*}
Thus, in particular, if the coefficients of the matrix $B$ are polynomial functions of order at most $1$ in the $u_{i'}$ variables, then $g$ is locally a
twisted Riemannian extension; a twisted Riemannian extension is locally a Riemannian extension if $B$ vanishes on 
the zero-section. In these two instances, the linear terms in the $u_{i'}$ variables give the connection $1$-form of a 
torsion-free connection on the base manifold.\\

The non-zero Christoffel symbols $\widetilde{\Gamma}_{\alpha \beta}^{\gamma}$ of the Levi-Civita connection of the twisted Riemannian extension
$g_{(\nabla,\phi)}$ are given by:
\begin{eqnarray*}
 \widetilde{\Gamma}^{k}_{ij} &=& \Gamma^{k}_{ij}, \quad \widetilde{\Gamma}^{k'}_{i'j} = -\Gamma^{i}_{jk}\quad
 \widetilde{\Gamma}^{k'}_{ij'} = -\Gamma^{j}_{ik},\\
 \widetilde{\Gamma}^{k'}_{ij} &=& \sum_r u_{r'} (\partial_k \Gamma^{r}_{ij} - \partial_i \Gamma^{r}_{jk}
 -\partial_j \Gamma^{r}_{ik} + 2 \sum_l \Gamma^{r}_{kl}\Gamma^{l}_{ij} )\\
 &\quad & + \frac{1}{2}(\partial_i \phi_{jk}+\partial_j \phi_{ik}-\partial_k \phi_{ij})
 - \sum_l \phi_{kl}\Gamma^{l}_{ij},
 \end{eqnarray*}
where $(i,j,k,l,r=1,\cdots,n)$ and $(i'=i+n,j'=j+n,k'=k+n,r'=r+n)$.
The non-zero components of the curvature tensor of $(T^*M,g_{(\nabla,\phi)})$ up to the usual symmetries are
given as follows: we omit $\widetilde{R}_{kji}^{h'}$, as it plays no role in our considerations.
\begin{equation*}
 \widetilde{R}_{kji}^{h}= R_{kji}^{h}, \;\;  \widetilde{R}_{kji}^{h'},\;\;\widetilde{R}_{kji'}^{h'}=-R_{kjh}^{i},\;\; \widetilde{R}_{k'ji}^{h'} = R_{hij}^{k},
\end{equation*} 
where $R_{kji}^{h}$ are the components of the curvature tensor of $(M,\nabla)$.\\

Twisted Riemannian extensions have nilpotent Ricci operator and hence, they are Einstein if and only if they are Ricci flat \cite{GarciaGilkeyNikcevicLorenzo}. They can be used to construct non-flat Ricci flat pseudo-Riemannian manifolds.

The classical and twisted Riemannian extensions provide a link between the affine geometry of $(M,\nabla)$ and the neutral signature metric
on $T^*M$. Some properties of the affine connection $\nabla$ can be investigated by means of the corresponding properties of the classical and twisted Riemannian extensions. For more details and information about classical Riemannian extensions and twisted Riemannian extensions, see \cite{broz,calvino1,calvino2,GarciaGilkeyNikcevicLorenzo,gar} and references therein.

\section{Szab\'o metrics on the cotangent bundle}\label{Szabo}

In this section, we recall some basic definitions and results on affine Szab\'o manifolds \cite{dm}. Using the classical and twisted Riemannian extensions, we exhibit some examples of pseudo-Riemannian Szab\'o metrics of signature $(3,3)$, which are not locally symmetric.

\subsection{The affine Szab\'o manifolds} 

Let $(M,\nabla)$ be an $n$-dimensional smooth affine manifold, where $\nabla$ is a torsion-free affine connection on $M$. Let $\mathcal{R}^{\nabla}$
be the associated curvature operator of $\nabla$. We define the \textit{affine Szab\'o operator}
$\mathcal{S}^{\nabla}(X):T_p M\rightarrow T_p M$ with respect to a vector $X\in T_p M$ by
$$
 \mathcal{S}^{\nabla}(X) Y := (\nabla_X \mathcal{R}^{\nabla})(Y,X)X.
$$
\begin{definition}\cite{dm} {\rm
Let $(M,\nabla)$ be a smooth affine manifold.
\begin{enumerate}
\item  $(M,\nabla)$ is called \textit{affine Szab\'o at $p\in M$} if the affine Szab\'o  operator $S^{\nabla}(X)$ has the 
same characteristic polynomial for every vector field $X$ on $M$.
\item Also, $(M,\nabla)$ is called \textit{affine Szab\'o} if $(M,\nabla)$ is affine Szab\'o at each point $p\in M$.
\end{enumerate}
}
\end{definition}

\begin{theorem}\cite{dm} 
Let $(M,\nabla)$ be an $n$-dimensional affine manifold and $p\in M$. Then $(M,\nabla)$ is affine Szab\'o at $p\in M$ if and only if the characteristic polynomial of the affine Szab\'o operator $S^{\nabla}(X)$ is $ P_{\lambda}(\mathcal{S}^{\nabla}(X))=\lambda^{n}$, for every $X\in T_{p}M$.
 \end{theorem}
This theorem leads to the following consequences which are proven in \cite{dm}.
\begin{corollary}\cite{dm} 
 $(M,\nabla)$ is affine Szab\'o if the affine Szab\'o operators are nilpotent, i.e., $0$ is the eigenvalue of $\mathcal{S}^{\nabla}(X)$ on the tangent bundle $T M$.   
\end{corollary} 
\begin{corollary}\cite{dm}  
 If $(M,\nabla)$ is affine Szab\'o at $p\in M$, then the Ricci tensor is cyclic parallel. 
\end{corollary}
Affine Szab\'o connections are well-understood in $2$-dimension, due to the fact that an affine connection is Szab\'o if and only if its Ricci tensor is cyclic parallel \cite{dm}. The situation is however more complicated in higher dimensions where the cyclic parallelism is a necessary but not sufficient condition for 
an affine connection to be Szab\'o.

According to Kowalski and Sekizawa \cite{KowalskiSekizawa}, an affine manifold $(M,\nabla)$ is said to be an $L_3$-space if its Ricci tensor is cyclic parallel. Then, we have:
\begin{theorem}
 Let $(M,\nabla)$ be a two-dimensional smooth torsion free affine manifold. Then, the following statements are equivalent:
 \begin{enumerate}
  \item $(M,\nabla)$ is an affine Szab\'o manifold.
  \item $(M,\nabla)$ is a $L_3$-space.
 \end{enumerate}
\end{theorem}
In higher dimensions, it is not hard to see that there exist $L_3$-spaces which are not affine Szab\'o manifolds.

Next, we have an example of a real smooth manifold of three dimensional in which the equivalence between Szab\'o and $L_{3}$ conditions holds. Let $M$ be a $3$-dimensional smooth manifold and $\nabla$ a torsion-free connection. We choose a fixed coordinates neighborhood $\mathcal{U}(u_1,u_2,u_3) \subset M$. 
\begin{proposition}\label{p1}
Let $M$ be a $3$-dimensional manifold with torsion free connection given by 
\begin{equation}\label{e1} 
\nabla_{\partial_1} \partial_1  =  f_1  \partial_2,\;\;
\nabla_{\partial_2} \partial_2  =  f_2  \partial_2,\;\;
\nabla_{\partial_3} \partial_3  =  f_3  \partial_2.  
\end{equation}
where $f_{i}= f_{i}(u_1,u_2,u_3)$, for $i=1,2,3$. Then $(M,\nabla)$ is affine Szab\'o if and only if the Ricci tensor of the affine connection (\ref{e1}) is cyclic parallel.
\end{proposition} 
\begin{proof}
 We denote the functions $f_1 (u_1,u_2,u_3), f_2 (u_1,u_2,u_3)$ and $f_1 (u_1,u_2,u_3)$ by $f_1,f_2$ and $f_3$
 respectively, if there is no risk of confusion. The Ricci tensor of the affine connection (\ref{e1}) expressed
 in the coordinates $(u_1,u_2,u_3)$ takes the form
 \begin{align}\label{e2} 
Ric^{\nabla} (\partial_1,\partial_1) &=  \partial_2 f_1 + f_1f_2,\quad
Ric^{\nabla} (\partial_1,\partial_2) = -\partial_1 f_2,\\
Ric^{\nabla} (\partial_3,\partial_2) &=  -\partial_3 f_2,\quad
Ric^{\nabla} (\partial_3,\partial_3) = \partial_2 f_3 + f_2f_3. 
\end{align}
It is know that the Ricci tensor of any affine Szab\'o is cyclic parallel \cite{dm}, it follows from the expressions in  
(\ref{e2}) that we have the following necessary condition for the affine connection (\ref{e1}) to be Szab\'o
\begin{align*}
 &\partial_1\partial_3 f_2  = 0,\\
 &\partial_1\partial_2 f_2 -f_2\partial_1 f_2  =0,\\
 &\partial_3\partial_2 f_2 -f_2\partial_3 f_2 =0,\\
 & \partial_1\partial_2 f_1 +2f_1\partial_1f_2 +f_2\partial_1f_1 =0,\\
 &\partial_2\partial_3 f_3 +2f_3\partial_3f_2 +f_2\partial_3f_3  =0,\\
 &\partial_2\partial_3 f_1 +2f_1\partial_3f_2 +f_2\partial_3f_1 =0,\\
 & \partial_2\partial_1 f_3 +2f_3\partial_1f_2 +f_2\partial_1f_3 =0,\\
& \partial^{2}_{2}f_1 +f_1\partial_2f_2 +f_2\partial_2f_1 -\partial^{2}_{1}f_2 =0,\\
& \partial^{2}_{2}f_3 +f_3\partial_2f_2 +f_2\partial_2f_3 -\partial^{2}_{3}f_2 =0.
\end{align*}
Now, for each vector $X=\sum_{i=1}^{3} \alpha_i \partial_i$, a straightforward calculation shows that the associated affine
Szab\'o operator is given by
\begin{eqnarray}\label{e3}
(\mathcal{S}^{\nabla} (X)) = 
\left(
      \begin{array}{lll}
      0&0&0\\
      a&0&c\\
       0&0&0\\
         \end{array}
\right),
\end{eqnarray}
with $a$ and $c$ are partial differential equations of $f_1,f_2$ and $f_3$. It follows from the matrix associated to 
$\mathcal{S}^{\nabla} (X)$, that its characteristic polynomial as written as follows:
$
 P_{\lambda} [\mathcal{R}^{\nabla} (X)] = \lambda^3.
$
It follows that a affine connection given by (\ref{e1}) is affine Szab\'o if its Ricci tensor is cyclic parallel.
\end{proof}

\begin{example}\label{Exampl1}{\rm
The following connection on $\mathbb{R}^3$ defined by
\begin{eqnarray}\label{e4}
\nabla_{\partial_1} \partial_1 = u_1 u_3 \partial_2, \quad
\nabla_{\partial_2} \partial_2 = 0, \quad 
\nabla_{\partial_3} \partial_3 = (u_1 + u_3) \partial_2
\end{eqnarray}
is a non-flat affine Szab\'o connection.}
\end{example}

\subsection{Szab\'o pseudo-Riemannian manifolds}

Let $(M,g)$ be a pseudo Riemannian manifold. The Szab\'o operator $$\mathcal{S}(X):Y\mapsto (\nabla_X R)(Y,X)X$$ is a symmetric operator with $\mathcal{S}(X)X=0$. It plays an important role in the study of totally isotropic manifols. Since $\mathcal{S}(\alpha X)=\alpha^3 \mathcal{S}(X)$, the natural domains of definition
for the Szab\'o operator are the pseudo-sphere bundles
$$
S^{\pm}(M,g) =\{X\in TM, g(X,X)=\pm 1\}.
$$
One says that $(M,g)$ is Szab\'o if the eigenvalues of $\mathcal{S}(X)$ are constant on the pseudo-spheres of unit timelike and spacelike vectors. The eigenvalue zero plays a distinguished role. One says that $(M,g)$ is nilpotent Szab\'o if $Spec(\mathcal{S}(X))=\{0\}$ for all $X$. If $(M,g)$ is nilpotent Szab\'o of order $1$, then $(M,g)$ is a local symmetric space (see \cite{fiedler} for more details).

Szab\'o in  \cite{Szabo} used techniques from algebraic topology to show, in the Riemannian setting, that any such a metric 
is locally symmetric. He used this observation to  prove that any two point homogeneous space 
is either flat or is a rank one symmetric space. Subsequently Gilkey and Stravrov \cite{GilkeyStravrov} 
extended this result to show that any Szab\'o Lorentzian manifold has constant sectional curvature.
However, for metrics of higher signature the situation is different. Indeed it was showed in \cite{GilkeyIvanovaZhang} 
the existence of Szab\'o pseudo-Riemannian manifolds endowed with metrics of signature $(p,q)$ 
with $p\geq 2$ and $q\geq 2$ which are not locally symmetric .

Next, we use the classical and twisted Riemannian extensions to construct some pseudo-Riemannian metrics on $\mathbb{R}^6$ which are nilpotent Szab\'o of order $\ge 2$.

\subsection{Riemannian extensions of an affine Szab\'o connection.}

We start with the following result.

\begin{theorem}
 Let $(M,\nabla)$ be a smooth torsion-free affine manifold. Then the following statements are equivalent:
 \begin{enumerate}
  \item[(i)] $(M,\nabla)$ is affine Szab\'o.
  \item[(ii)]  The Riemannian extension $(T^*M,g_{\nabla})$ of $(M,\nabla)$ is a pseudo Riemannian Szab\'o manifold.
 \end{enumerate}
\end{theorem}
\begin{proof}
 Let $\tilde{X}=\alpha_i \partial_i + \alpha_{i'}\partial_{i'}$ be a vector field on $T^* M$. 
 Then the matrix of the Szab\'o operator $\tilde{S}(\tilde{X})$ with respect to the basis $\{\partial_i,\partial_{i'}\}$ is of the form
 \begin{eqnarray} 
 \tilde{\mathcal{S}}(\tilde{X}) = \left(\begin{array}{cc}
                                         \mathcal{S}^{\nabla}(X)&0\\
                                         *& {}^t\mathcal{S}^{\nabla}(X)
                                        \end{array}
\right).\nonumber
\end{eqnarray}
where $\mathcal{S}^{\nabla}(X)$ is the matrix of the affine Szab\'o operator on $M$ relative
to the basis $\{\partial_i\}$. Note that the characteristic polynomial $P_{\lambda}[\tilde{\mathcal{S}}(\tilde{X})]$ of $\tilde{\mathcal{S}}(\tilde{X})$
and $P_{\lambda}[\mathcal{S}^{\nabla}(X)]$ of $\mathcal{S}^{\nabla}(X)$ are related by 
\begin{eqnarray}\label{SzaboMatrix} 
P_{\lambda}[\tilde{\mathcal{S}}(\tilde{X})]=P_{\lambda}[\mathcal{S}^{\nabla}(X)]\cdot P_{\lambda}[{}^t\mathcal{S}^{\nabla}(X)]. 
\end{eqnarray}
Now, if the affine manifold $(M,\nabla)$ is assumed to be affine Szab\'o, then $\mathcal{S}^{\nabla}(X)$ has zero eigenvalues for each
vector field $X$ on $M$. Therefore, it follows from (\ref{SzaboMatrix}) that the eigenvalues of $\tilde{\mathcal{S}}(\tilde{X})$ vanish
for every vector field $\tilde{X}$ on $T^* M$. Thus $(T^*M, g_{\nabla})$ is pseudo-Riemannian Szab\'o manifold.\\
Conversely, assume that $(T^*M, g_{\nabla})$ is an pseudo-Riemannian Szab\'o manifold. If $X=\alpha_i \partial_i$ is an arbitrary
vector field on $M$ then $\tilde{X}=\alpha_i \partial_i + \frac{1}{2\alpha_i}\partial_{i'}$ is an unit vector field at every point of
the zero section on $T^* M$. Then from (\ref{SzaboMatrix}), we see that, the characteristic polynomial 
$P_{\lambda}[\tilde{\mathcal{S}}(\tilde{X})]$ of $\tilde{\mathcal{S}}(\tilde{X})$ is the square of the characteristic polynomial
$P_{\lambda}[\mathcal{S}^{\nabla}(X)]$ of $\mathcal{S}^{\nabla}(X)$. Since for every unit vector field $\tilde{X}$ on $T^* M$ the
characteristic polynomial $P_{\lambda}[\tilde{\mathcal{S}}(\tilde{X})]$ should be the same, it follows that for every vector field
$X$ on $M$ the characteristic polynomial $P_{\lambda}[\mathcal{S}^{\nabla}(X)]$ is the same. Hence $(M,\nabla)$ is affine Szab\'o.
\end{proof}

As an example, we have the following. Let $(M,\nabla)$ be a $3$-dimensional affine manifold. Let $(u_1,u_2,u_3)$ be local coordinates on $M$. We write
$\nabla_{\partial_i} \partial_j = \sum_k \Gamma_{ij}^{k}\partial_k$ for $i,j,k=1,2,3$ to define the coefficients of affine connection $\nabla$. If $\omega \in T^* M$, we write $\omega=u_4du_1 + u_5du_2 + u_6du_3$ to define the dual fiber coordinates $(u_4,u_5,u_6)$, and thereby obtain canonical local coordinates $(u_1,u_2,u_3,u_4,u_5,u_6)$ on $T^* M$. 
The Riemannian extension is the metric of neutral signature $(3,3)$ on the cotangent bundle $T^* M$ locally given by
\begin{eqnarray*}
 g_{\nabla}(\partial_1,\partial_4)&=& g_{\nabla}(\partial_2,\partial_5)=g_{\nabla}(\partial_3,\partial_6)=1,\\
 g_{\nabla}(\partial_1,\partial_1)&=& -2u_4\Gamma_{11}^{1}-2u_5\Gamma_{11}^{2}-2u_6\Gamma_{11}^{3},\\
 g_{\nabla}(\partial_1,\partial_2)&=& -2u_4\Gamma_{12}^{1}-2u_5\Gamma_{12}^{2}-2u_6\Gamma_{12}^{3},\\
 g_{\nabla}(\partial_1,\partial_3)&=& -2u_4\Gamma_{13}^{1}-2u_5\Gamma_{13}^{2}-2u_6\Gamma_{13}^{3},\\
 g_{\nabla}(\partial_2,\partial_2)&=& -2u_4\Gamma_{22}^{1}-2u_5\Gamma_{22}^{2}-2u_6\Gamma_{22}^{3},\\
 g_{\nabla}(\partial_2,\partial_3)&=& -2u_4\Gamma_{23}^{1}-2u_5\Gamma_{23}^{2}-2u_6\Gamma_{23}^{3},\\
 g_{\nabla}(\partial_3,\partial_3)&=& -2u_4\Gamma_{33}^{1}-2u_5\Gamma_{33}^{2}-2u_6\Gamma_{33}^{3}.
 \end{eqnarray*}
 From Example \ref{Exampl1}, the Riemannian extension of the affine connection defined in (\ref{e4}) is the pseudo-Riemannian metric given by
 \begin{eqnarray}\label{metric1}
 g&=& 2 du_1\otimes du_4 + 2 du_2\otimes du_5 + 2du_3\otimes du_6 \nonumber\\
 &\quad & -2(u_1u_3u_5)du_1\otimes du_1 -2(u_1 + u_3)u_5du_3\otimes du_3.
\end{eqnarray}
This metric leads to the following result. 
\begin{proposition}
The metric in (\ref{metric1}) is Szab\'o of signature $(3,3)$ with zero eigenvalues. Moreover,
it is not locally symmetric.
\end{proposition} 
 \begin{proof}
 The non-vanishing components of the curvature tensor of $(\mathbb{R}^6,g_{\nabla})$ are given by
 $R(\partial_1, \partial_3 )\partial_1  =   - u_1\partial_2$, $R(\partial_1, \partial_3 )\partial_3 = \partial_2$, $R(\partial_1,\partial_3)\partial_5 = u_1\partial_4 - \partial_6$, $R(\partial_1, \partial_5 )\partial_1  =  - u_1 \partial_6$, $R(\partial_1, \partial_5 )\partial_3 = u_1 \partial_4$, $R(\partial_3, \partial_5)\partial_1 = \partial_6$, $ R(\partial_3, \partial_5 )\partial_3  =  - \partial_4$.
Let $\displaystyle X = \sum_{i=1}^{6}\alpha_i\partial_i$ be a non-zero vector on $\mathbb{R}^6$. Then the matrix associated with the Szab\'o operator $\mathcal{S} (X) := (\nabla_X \mathcal{R})(\cdot,X)X$ is given by
\begin{equation*}
(\mathcal{S} (X)) = 
\left(
      \begin{array}{cccccc}
      0&0&0&0&0&0\\
      -1&0&1&0&0&0\\
      0&0&0&0&0&0\\
      2&0&-1&0&-1&0\\
      0&0&0&0&0&0\\
      -1&0&1&0&1&0
       \end{array}
\right).
\end{equation*}
Hence the characteristic polynomial of the Szab\'o operators is $P_{\lambda}(S(X))=\lambda^6$. Since one of the components of $\nabla R$,
$
 (\nabla_{\partial_1} R)(\partial_1,\partial_3,\partial_5,\partial_1) =1
$
is non-zero, the metric in (\ref{metric1}) is not locally symmetric. The proof is completed.
\end{proof}

\subsection{Twisted Riemannian extensions of an affine Szab\'o connection}

In this subsection, we study the twisted Riemannian extensions which is a generalization of classical Riemannian extensions. We have following result.

\begin{theorem} 
Let $(T^*M,g_{\nabla,\phi})$ be the cotangent bundle of an affine manifold $(M,\nabla)$ equipped with the 
twisted Riemannian extension. 

Then $(T^*M,g_{(\nabla,\phi)})$ is a pseudo-Riemannian Szab\'o manifold if and
only $(M,\nabla)$ is affine Szab\'o for any symmetric $(0,2)$-tensor field $\phi$.
\end{theorem}
As an example we have the following.
\begin{example}{\rm
Let us consider the twisted Riemannian extensions of the affine connection $\nabla$ given in Example \ref{Exampl1}. This is given by  
\begin{align}\label{metric2}
 g&=  2du_1\otimes du_4 + 2du_2\otimes du_5 + 2du_3\otimes du_6  + 2\phi_{12}du_1\otimes du_2 \nonumber\\
 &+ 2\phi_{13}du_1\otimes du_3+ 2\phi_{23}du_2\otimes du_3 +(\phi_{11} -2u_1u_3u_5)du_1\otimes du_1\nonumber\\
 & +\phi_{22}du_2\otimes du_2   +[\phi_{33}-2(u_1 + u_3)u_5]du_3\otimes du_3, 
\end{align}
 where $(u_1,u_2,\cdots,u_6)$ are coordinates in $\mathbb{R}^6$.
The non-zero Christoffel symbols are as follows:
 \begin{align}
 \Gamma^{2}_{11} &= -\Gamma^{4}_{15}= u_{1}u_{3}, \;\;
 \Gamma^{4}_{11} = \frac{1}{2}\partial_{1}\phi_{11}-u_{3}u_{5}-u_{1}u_{3}\phi_{12},\nonumber\\
 \Gamma^{5}_{11} &=  \partial_{1}\phi_{12}-\frac{1}{2}\partial_{2}\phi_{11}-u_{1}u_{3}\phi_{22},\;\;
 \Gamma^{6}_{11} = \partial_{1}\phi_{13}+u_{5}u_{1}-\frac{1}{2}\partial_{3}\phi_{11}-u_{1}u_{3}\phi_{22},\nonumber\\ 
 \Gamma^{4}_{12} &=  \frac{1}{2}\partial_{2}\phi_{11},\;\;
 \Gamma^{5}_{12} = \frac{1}{2}\partial_{1}\phi_{22},\;\;
 \Gamma^{6}_{12} = \frac{1}{2}\{\partial_{2}\phi_{13}+\partial_{1}\phi_{23}-\partial_{3}\phi_{12}\},\nonumber\\  
 \Gamma^{4}_{13} &=  -u_{5}u_{1} + \frac{1}{2}\partial_{3}\phi_{11},\;\;
 \Gamma^{5}_{13} = \frac{1}{2}\{\partial_{3}\phi_{12}+\partial_{1}\phi_{32}-\partial_{2}\phi_{13}\},\nonumber\\  
 \Gamma^{6}_{13} &= -u_{5}+\frac{1}{2}\partial_{1}\phi_{33},\;\;
 \Gamma^{4}_{22}  =  \partial_{2}\phi_{21}-\frac{1}{2}\partial_{1}\phi_{22},\;\;
 \Gamma^{5}_{22} = \frac{1}{2}\partial_{2}\phi_{22},\nonumber \\ 
 \Gamma^{6}_{22}  &= \partial_{2}\phi_{23}-\frac{1}{2}\partial_{3}\phi_{22}, 
 \Gamma^{4}_{23}  =  \frac{1}{2}\{\partial_{3}\phi_{21}+\partial_{2}\phi_{31}-\partial_{1}\phi_{23}\},\nonumber
 \end{align}
  \begin{align} 
 \Gamma^{5}_{23}& = \frac{1}{2}\partial_3 \phi_{22},\;\;
 \Gamma^{6}_{23}  = \frac{1}{2}\partial_{2}\phi_{33}, \;\;
 \Gamma^{2}_{33}  = -\Gamma^{6}_{35} = (u_{1}+u_{3}),\nonumber\\  
 \Gamma^{4}_{33}& = \partial_{3}\phi_{31}+u_{5}-\frac{1}{2}\partial_{1}\phi_{33}-(u_{1}+u_{3})\phi_{12},\nonumber\\
 \Gamma^{5}_{33} &=  \partial_{3}\phi_{32}-\frac{1}{2}\partial_{2}\phi_{33}-(u_{1}+u_{3})\phi_{22},\;\;
 \Gamma^{5}_{33} = -u_{5} + \partial_{3}\phi_{33} - (u_{1} + u_{3})\phi_{32}.\nonumber
 \end{align}
For $\displaystyle X=\sum_{i=1}^{6}\alpha_i\partial_i$, by a straightforward calculation the characteristic polynomial associated with the Szab\'o operator is $P_{\lambda}[S(X)]=\lambda^6$. So, $(M,g_{\nabla,\phi})$ is a pseudo-Riemannian Szab\'o metric of signature $(3,3)$ with zero eigenvalue.}
\end{example}

\section*{Acknowledgments}

The authors would like to thank the referee for his/her valuable suggestions and comments that helped them improve the paper.

\end{document}